\documentclass[reqno, 12pt]{amsart}
\usepackage{enumerate}
\usepackage{amssymb,url,color}
\usepackage{hyperref}
\usepackage{amsmath,amsthm}

\allowdisplaybreaks[4]

\newtheorem{thm}{Theorem}[section]
\newtheorem{cor}{Corollary}[section]

\newtheorem{lem}{Lemma}[section]
\newtheorem{rem}{Remark}[section]

\theoremstyle{Problem}

\theoremstyle{definition}

\numberwithin{equation}{section}
\newcommand{\pp}{\mathbb{P}}
\newcommand{\ee}{\mathbb{E}}

\newcommand{\FF}{\mathcal{F}}

\newcommand{\rr}{\mathbb{R}}

\def\beq{\begin{equation}}
\def\deq{\end{equation}}
\allowdisplaybreaks 

\bfseries\rmfamily 

\begin{document}

\title[Berry-Esseen bound of modularity in network]
{Berry-Esseen bound of modularity in network}
\thanks{This work is supported by National Natural Science Foundation of China (NSFC-11971154).}

\author[Y. Miao]{Yu Miao}
\address[Y. Miao]{College of Mathematics and Information Science, Henan Normal University, Henan Province, 453007, China; Henan Engineering Laboratory for Big Data Statistical Analysis and Optimal Control, Henan Normal University, Henan Province, 453007, China.} \email{\href{mailto: Y. Miao
<yumiao728@gmail.com>}{yumiao728@gmail.com}; \href{mailto: Y. Miao <yumiao728@126.com>}{yumiao728@126.com}}

\author[Q. Yin]{Qing Yin}
\address[Q. Yin]{College of Mathematics and Information Science, Henan Normal University, Henan Province, 453007, China.}
\email{\href{mailto: Q. Yin
<qingyin1282@163.com>}{qingyin1282@163.com}}

\begin{abstract} In this paper, the model is a specific partition of a given network. Berry-Esseen bound and strong law of large numbers of modularity for the partition are proved when the size of the network gets large.
\end{abstract}


\maketitle

\section{Introduction}
Networks have been the focus of much recent attention since they describe a multitude of complex systems found in many fields. Existing networks often display a high level of local inhomogeneity, with high edge density within certain groups of nodes and low edge density between these groups.
The desire to divide the network into communities, one of the most relevant features representing real systems is community structure.

Due to the importance of finding community structures in networks, there has been work on this topic in  such fields as computer science, physics, statistics, sociology, and many others (see \cite{A-B,Jackson,Newman,N-10}). Fortunato \cite{Fortunato} presented some striking examples of real networks with community structure. In this way they saw what communities look like and why they are important. After the detection of communities, it is an important issue to assess their statistical significance.
In order to distinguish meaningful structural changes from random fluctuations, Rosvall and Bergstrom \cite{R-B} provided a solution to this problem by using bootstrap resampling accompanied by significance clustering. Lancichinetti et al. \cite{L-F-R} introduced a measure aimed at quantifying the statistical significance of single communities.
Zhang and Chen \cite{Z-C} re-examined the null model in the Newman-Girvan modularity function and provide a statistical framework for modularity-based community detection. Based on it, they introduced a hypothesis testing procedure to determine the significance of the partitions obtained from maximizing the modularity function. They showed that the modularity formulated under our framework is consistent under a degree-corrected stochastic block model framework.
Ma and Barnett \cite{M-B} proved that the largest eigenvalue and modularity are asymptotically uncorrelated, which suggests the need for inference directly on modularity itself when the network is large. Weighted networks with signed edges such as correlation networks can be well-modelled by Gaussian orthogonal ensemble random matrices under a variety of null models.

Li and Qi \cite{L-Q} proposed a way of evaluating the significance of any given partition by considering whether this particular partition can arise simply from randomness under the assumption that there is no underlying community structure in the network. They established a specific partition of a given network and established that the distribution of modularity under a null hypothesis of free labeling is asymptotically normal when the size of the network gets large. The significance of the partition is defined based on this asymptotic distribution, which can help assess its goodness. Two different partitions can also be compared statistically. Simulation studies and real data analyses are performed for illustration. The model for a specific partition of a given network is as follows.

Consider an undirected graph $G$ consisting of $n$ vertices $\{v_1, v_2, \cdots, v_n\}$ and $m$ edges $\{e_1, e_2, \cdots, e_m\}$. Let $k_i(n)$ denote the degree of vertex $v_i$, which is the number of edges connected to vertex $v_i$. In order to simplify the notation, we write $k_i$ instead of $k_i(n)$. Then it holds that
$$
\sum_{i=1}^{n}k_i=2m, \ \ \ \ \text{for} \ \ 1\le i\le n.
$$
Let $A_{ij}$ be the number of edges between vertex $v_i$ and vertex $v_j$, for $1\le i, j\le n$.
In the paper, we discuss a simple graph, for which $A_{ij}$ is $0$ or $1$, and $A_{ii}=0$.
So we have
$$
k_i=\sum_{j=1}^{n}A_{ij}=\sum_{j=1}^{n}A_{ji}, \ \ \ \text{for} \ \ \ 1\leq i\leq n.
$$

Let $C$ denote a partition of network $G$ (using the existing community detection method, see Fortunato \cite{Fortunato}), i.e., each vertex $v_i\ (1\leq i\leq n)$ is associated with a group label or color $c_i\in \{1, 2, \cdots, K\}$, where $K$ is the total number of communities by the partition, and we denote $C=(c_1, c_2, \cdots, c_n)$.

Newman \cite{N-06} introduced the following modularity of the partition $C$,
\beq\label{1-1}
Q_n(C)=\frac{1}{2m}\sum_{i,j}\left(A_{ij}-\frac{k_ik_j}{2m}\right)\delta_{c_{i},c_{j}}
=\frac{1}{2m}\sum_{i,j}B_{ij}\delta_{c_{i},c_{j}},
\deq
where $\delta_{c_{i},c_{j}}$ is the Kronecker delta function which takes value 1 if vertices $i$ and $j$ are in the same group, i.e., $c_i=c_j$, and zero otherwise. In addition,
\beq\label{h1}
B_{ij}=A_{ij}-\frac{k_ik_j}{2m},\ \ \ \ 1\le i,j\le n.
\deq
It is not difficult to check that $-1<Q_n(C)<1$, and $Q_n(C)$ is the weighted sum of $B_{ij}$ over all pairs of vertices $i$, $j$ that fall in the same groups. It measures the extent to which vertices of the same type are connected to each other in a network.

For a given partition $C$ of the network, we are interested in whether this partition could be obtained by randomly assigning colors to the vertices. The global null hypothesis $H_0$ is that the colors are assigned to vertices randomly, regardless of the structure of the network. The probability that a given vertex is labeled as group $1$ is $p_1=|Col(1)|/n$, where $Col(1)$ is the cardinality of the set of vertices with color $1$; the probability is $p_2=|Col(2)|/n$ for group $2$, and so on. For any $1\leq k\leq K$, it is easy to check that
$$
p_1+p_2+\cdots+p_K=1, \ \ \ \ p_k\ge0.
$$
The labeling of different vertices is assumed to be independent so $H_0$ is also called {\em free labeling}.

Assume that the partition $C=(c_1,c_2,\ldots,c_n)$ is a random vector, where $c_1,c_2,\ldots,c_n$ are independent identically distribution random variables, and have the following distribution
$$
\pp(c_i=j)=p_j, \ \ \ \ 1\leq j\leq K.
$$
Denote
\beq\label{e-e}
p_{(l)}=\sum_{k=1}^{K}p_k^l, \ \ \ \text{for} \ \ \ l=1, 2, \cdots
\deq
and
\beq\label{b-1}
\bar{h}(c_i,c_j)=\delta_{c_i,c_j}-p_{c_{i}}-p_{c_{j}}+p_{(2)}, \ \ \ 1\le i\neq j\le n.
\deq
In this case, we denote $Q_n(C)$ by $Q_n$ to avoid confusion. Li and Qi \cite{L-Q} proved the following asymptotic normality of $Q_n$ under some conditions:
 \begin{align}\label{clt}
 \frac{Q_n-\mu_n}{\sigma_n}\xrightarrow{d}N(0,1),
 \end{align}
 where $\mu_n$ and $\sigma_n^2$ are given by
\beq\label{2-3}
\mu_n=\ee[Q_n]=-\frac{1-p_{(2)}}{4m^2}\sum_{i=1}^nk_i^2,
\deq
\beq\label{2-3-1}
\sigma_n^2=Var(Q_n)=\frac{p_{(2)}+p_{(2)}^2-2p_{(3)}}{2m^2}\sum_{1\leq i\neq j\leq n}B_{ij}^2+\frac{p_{(3)}-p_{(2)}^2}{m^2}\sum_{i=1}^nB_{ii}^2.
\deq
Yin et al. \cite{M-Y-W-Y} proved the moderate deviation principle of the modularity estimator for the specific partition of a given network. Based on the above results, we are interested in the Berry-Esseen bound of modularity in network. The Berry-Esseen theorem and its extensions are of great significance in probability and statistics. The following is one version of the celebrated Berry-Esseen theorem, discovered by Berry \cite{Berry} and Esseen \cite{Esseen}.
\begin{thm}\label{thm1-1}
Let $\{X_i,1\le i\le n\}$ be a sequence of independent random variable with $\ee X_n=0,\ \ee X_n^2=1$ and bounded third moments: $\sup_{1\le n<\infty}\ee|X_n|^3\le \rho$. Then for all $n$,
$$
\sup_{-\infty<x<\infty}\Bigg|\pp\left(n^{-1/2}\sum_{i=1}^nX_i\le x\right)-\Phi(x)\Bigg|\le C n^{-1/2}\rho,
$$
where $\Phi$ denotes the standard normal distribution function and $C$ is an absolute constant.
\end{thm}
There is a very extensive literature relating to rates of convergence in the central limit theorem for sums of independent random variables. Comprehensive accounts are given in, for example, Gnedenko and Kolmogorov \cite{G-K}, Ibragimov and Linnik \cite{I-L} and Petrov \cite{Petrov}. The results provide a neat and accurate estimate of the error term in the statistician's normal approximations, and the rate of convergence of order $n^{-1/2}$ is often fast enough to justify his testing procedures. However, a fast rate of convergence can only be achieved by imposing some type of restriction on the condition variances. Hall and Heyde \cite{H-H} established a rate of convergence of almost $n^{-1/4}$.
In addition, Jiang \cite{Jiang} established Berry-Esseen bound for martingale array.
His result on the uniform convergence rate is the same as Hall and Heyde \cite{H-H}, but presented in the form of a martingale array.

The paper is organized as follows, we study the Berry-Esseen bound and strong law of large numbers of modularity in network when the size of the network gets large. Our approach is based on Berry-Esseen bound for martingale array due to Jiang \cite{Jiang}. Our main results are stated and discussed in Section 2. In Section 3, we state some important lemmas that will be used in this paper. Section 4 contains the proofs of main results. The symbol $H$ denotes a positive constant which is not necessarily the same one in each appearance.

\section{Main results}
In this section, we state the main results of the paper.
Denote
\beq\label{b1}
\delta_n=\left(\frac{p_{(2)}+p_{(2)}^2-2p_{(3)}}{m}\right)^{1/2}.
\deq

\begin{thm}\label{je}
Assume that the degree sequence $\{k_i,1\le i\le n\}$ satisfies the following conditions:
\beq\label{3-1}
\frac{\max_{1\le i\le n}k_i}{\sqrt{m}}\le Hn^{-1/2}
\deq
and
\beq\label{3-1-1}
\frac{1}{m^2}\sum_{1\le i,j\le n}\left(\sum_{l=1}^nA_{il}A_{jl}\right)^2\le Hn^{-5/4}\left(\log n\right)^{5}.
\deq
Then for $n\geq2$, we have
$$
\sup_{-\infty<x<\infty}\left|\pp\left(\frac{Q_n-\mu_n}{\delta_n}\leq x\right)-\Phi(x)\right|
\le Hn^{-1/4}\log n.
$$
\end{thm}
\begin{rem}\label{rem3-1}
If (\ref{3-1}) and (\ref{3-1-1}) hold, then we have
$$
\lim_{n\rightarrow\infty}\frac{1}{m^3}\left(\sum_{i=1}^nk_i^2\right)^2=0
$$
and
$$
\lim_{n\rightarrow\infty}\frac{1}{m^2}\sum_{1\le i,j\le n}\left(\sum_{l=1}^nA_{il}A_{jl}\right)^2=0.
$$
Therefore, the conclusion (\ref{clt}) holds if (\ref{3-1}) and (\ref{3-1-1}) are satisfied (see (2.6) of \cite{M-Y-W-Y} and Theorem 1 of Li and Qi \cite{L-Q}).
\end{rem}
\begin{rem}\label{rem2-11}
A sufficient condition for (\ref{3-1-1}) is
\beq\label{c-1}
\frac{\max_{1\le i\le n}k_i}{\sqrt{m}}\le n^{-5/8}\left(\log n\right)^{5/2}.
\deq
In fact, we deduce
$$
\aligned
\frac{1}{m^2}\sum_{1\le i,j\le n}\left(\sum_{l=1}^nA_{il}A_{jl}\right)^2
&\le\frac{1}{m^2}\sum_{1\le i,j\le n}k_i\sum_{l=1}^nA_{il}A_{jl}\\
&=\frac{1}{m^2}\sum_{1\le i\le n}\sum_{1\le l\le n}k_iA_{il}\sum_{1\le j\le n}A_{jl}\\
&\le \frac{1}{m^2}\max_{1\le j\le n}k_j \sum_{1\le i\le n}k_i\sum_{1\le l\le n}A_{il}\\
&\le \frac{2\max_{1\le j\le n}k_j^2}{m}.
\endaligned
$$
If
$$
\frac{\max_{1\le i\le n}k_i}{\sqrt{m}}\le n^{-5/8}\left(\log n\right)^{5/2},
$$
then (\ref{3-1-1}) holds. Moreover, there exists a positive integer $N_0$, for any $n>N_0$, we conclude that (\ref{c-1}) is also sufficient for (\ref{3-1}).
\end{rem}

\begin{rem}\label{rem3-3}
For the network $G$ considered in the present paper, since $A_{ii}=0$ and $A_{ij}\in\{0,1\}$ for $i\neq j$ we have
$$
0\le m\le n(n-1)/2\ \ \text{and}\ \ 0\le \max\limits_{1\le i\le n}k_i\le n-1.
$$
Notice that, if $\max\limits_{1\le i\le n}k_i\sim \sqrt{n}$, then
$$
\frac{\max_{1\le i\le n}k_i}{\sqrt{m}}\ge \frac{H\sqrt{n}}{\sqrt{n(n-1)}}\ge Hn^{-1/2}\ge H n^{-5/8}\left(\log n\right)^{5/2},
$$
so the conditions (\ref{3-1}) and (\ref{c-1}) do not hold in this case.
\end{rem}

\begin{cor}\label{cor2-1}
Assume that the conditions (\ref{3-1}) and (\ref{3-1-1}) are satisfied. Then for $n\geq2$, we have
$$
\sup_{-\infty<x<\infty}\left|\pp\left(\frac{Q_n-\mu_n}{\sigma_n}\leq x\right)-\Phi(x)\right|
\le Hn^{-1/4}\log n.
$$
\end{cor}

\begin{thm}\label{thm2-2}
Assume that $\{b_n,n\ge1\}$ is a sequence of positive constants such that
\beq\label{p}
\frac{b_n\log n}{\sqrt{m}}\to 0 \ \ \ \text{as}\ \ \ n\to\infty
\deq
and the conditions (\ref{3-1}) and (\ref{3-1-1}) hold.
Then we have
\beq\label{e}
b_n\left(Q_n-\mu_n\right)\xrightarrow{a.s.}0.
\deq
\end{thm}

\section{Preliminary lemmas}
In the proofs of main results, we make use of the following lemmas.

\begin{lem}\label{lem2-2}$($\cite[(2.18)]{G-L-Z}$)$
Let $\{X_i,1\le i\le n\}$ be a sequence of independent identically distributed random variables, and for any $1\le i,j\le n$, $h_{i,j}(u,v): \rr^2\to \rr$ be measurable and symmetric with respect to its arguments. For any $1\le i\ne j\le n$, assume that $\ee(h_{i,j}(X_i,X_j)|X_j)=0$, $\ee(h_{i,j}(X_i,X_j)|X_i)=0$ and $\ee|h_{i,j}(X_1,X_2)|^p<\infty$ for some $p\ge 2$. Then we have
$$
\ee\left|\sum_{1\le i< j \le n}h_{i,j}(X_i,X_j)\right|^p\le 4^{p}p^p\ee \left(\sum_{1\le i< j \le n}h_{i,j}^2(X_i,X_j)\right)^{p/2}.
$$
\end{lem}

\begin{lem}\label{lem2-3-1}
Under the conditions in Lemma \ref{lem2-2}, if there exist positive constants $\{a_{i,j}, 1\le i,j\le n\}$, such that
$|h_{i,j}(u,v)|\le a_{i,j}$ for all $1\le i,j\le n$ and all $u,v \in\rr$, then for any $x>8eD_n$, we have
$$
\pp\left(\bigg|\sum_{1\le i< j \le n}h_{i,j}(X_i,X_j)\bigg|>x\right)\le \exp\left\{-\frac{x}{4eD_n}\right\},
$$
where
$$
D_n= \left(\sum_{1\le i< j \le n}a_{i,j}^2\right)^{1/2}.
$$
\end{lem}
\begin{proof} From Lemma \ref{lem2-2}, we have
$$
\ee\left|\sum_{1\le i< j \le n}h_{i,j}(X_i,X_j)\right|^p\le 4^{p}p^p \left(\sum_{1\le i< j \le n}a_{i,j}^2\right)^{p/2}=:(4D_n)^{p}p^p.
$$
Let $p=x/(4eD_n)$ for any $x>8eD_n$. Then, by the Markov's inequality, we have
$$
\pp\left(\bigg|\sum_{1\le i< j \le n}h_{i,j}(X_i,X_j)\bigg|>x\right)\le \exp\left\{-\frac{x}{4eD_n}\right\}.
$$
\end{proof}

\begin{lem}\label{lem5} $($\cite[Lemma 2]{C-R}$)$
For any random variables $X$, $Y$, a real number $x$ and a constant $a>0$,
$$
\sup_{x}\big|\pp(X+Y\leq x)-\Phi(x)\big|\leq\sup_{x}\big|\pp(X\leq x)-\Phi(x)\big|+\frac{a}{\sqrt{2\pi}}+\pp\left(|Y|>a\right),
$$
where $\Phi(x)$ is the standard normal distribution.
\end{lem}

\begin{lem}\label{lem2-6} $($\cite[Theorem 8.8]{Jiang}$)$
Let $\{S_{ni}=\sum_{j=1}^iX_{nj},\ \FF_{ni},\ 1\leq i\leq n\}$ be an array of martingales, where $\FF_{ni}=\sigma(X_{n1},X_{n2},\ldots,X_{ni})$. Let
$$
V_{ni}^2=\sum_{j=1}^i\ee\left(X_{nj}^2|\FF_{n,j-1}\right), \ \ \ 1\le i\le n.
$$
Write $S_n=S_{nn}$ and $V_n^2=V_{nn}^2$. Suppose that
\beq\label{2-6-1}
\max_{1\le i\le n}|X_{ni}|\leq Mn^{-1/2} \ \ a.s.
\deq
and
\beq\label{2-8}
\pp\left(|V_{n}^2-1|>9M^2Dn^{-1/2}(\log n)^2\right)\leq Hn^{-1/4}\log n
\deq
for constants $M,\ H,$ and $D\ (\geq e)$. Then for $n\geq2$,
$$
\sup_{-\infty<x<\infty}\big|\pp(S_{n}\leq x)-\Phi(x)\big|\le\left(2+H+7MD^{1/2}\right)n^{-1/4}\log n,
$$
where $\Phi(x)$ is the standard normal distribution.
\end{lem}

\section{Proofs of main results}

\begin{proof}[Proof of Theorem \ref{je}]
Since
$$
\sum_{i=1}^nB_{ij}=\sum_{j=1}^nB_{ij}=0,
$$
we deduce that
$$
\sum_{1\leq i\neq j\leq n}B_{ij}=\sum_{1\leq i,j\leq n}B_{ij}-\sum_{i=1}^nB_{ii}=-\sum_{i=1}^nB_{ii}.
$$
Note that by (\ref{1-1}) and (\ref{b-1}), we have
\beq\label{cc}
\aligned
Q_n&=\frac{1}{2m}\sum_{i=1}^nB_{ii}+\frac{1}{2m}\sum_{1\leq i\neq j\leq n}B_{ij}\bar{h}(c_i,c_j)\\
&\ \ \ \ +\frac{1}{2m}\sum_{1\leq i\neq j\leq n}B_{ij}(p_{c_{i}}+p_{c_{j}})-\frac{p_{(2)}}{2m}\sum_{1\leq i\neq j\leq n}B_{ij}\\
&=\frac{1+p_{(2)}}{2m}\sum_{i=1}^nB_{ii}+\frac{1}{m}\sum_{1\leq i<j\leq n}B_{ij}\bar{h}(c_i,c_j)-\frac{1}{m}\sum_{i=1}^nB_{ii}p_{c_{i}}\\
&=\frac{1-p_{(2)}}{2m}\sum_{i=1}^nB_{ii}+\frac{1}{m}\sum_{1\leq i<j\leq n}B_{ij}\bar{h}(c_i,c_j)-\frac{1}{m}\sum_{i=1}^nB_{ii}(p_{c_{i}}-p_{(2)}).
\endaligned
\deq
Combining (\ref{2-3}), (\ref{b1}) and the above equation together, it holds that
\beq\label{pp-1}
\aligned
\frac{Q_n-\mu_n}{\delta_n}
&=\frac{1}{m\delta_n}\sum_{1\leq i<j\leq n}B_{ij}\bar{h}(c_i,c_j)-\frac{1}{m\delta_n}\sum_{i=1}^nB_{ii}(p_{c_{i}}-p_{(2)})\\
&=\frac{1}{m\delta_n}\sum_{1\leq i<j\leq n}A_{ij}\overline{h}(c_i,c_j)-\frac{1}{2m\delta_n}\sum_{1\leq i<j\leq n}\frac{k_ik_j}{m}\overline{h}(c_i,c_j)\\
&\ \ \ +\frac{1}{2m\delta_n}\sum_{i=1}^n\frac{k_i^2}{m}(p_{c_{i}}-p_{(2)}).
\endaligned
\deq
Let
\beq\label{pp-2}
T_n=\frac{1}{m\delta_n}\sum_{j=1}^n\sum_{i=1}^{j-1}A_{ij}\overline{h}(c_i,c_j)=\frac{1}{m\delta_n}\sum_{j=1}^nz_{nj},
\deq
where $z_{n1}=0$ and $z_{nj}=\sum_{i=1}^{j-1}A_{ij}\overline{h}(c_i,c_j)$ for $2\le j\le n$.
Let $\FF_j=\sigma(c_1,c_2,\ldots,c_j)$ denote the $\sigma$-algebra generated by $\{c_1,c_2,\ldots,c_j\}$ for $1\leq j\leq n$. From (\ref{e-e}) and (\ref{b-1}), for any $1\le i<j\le n$, it is easy to check that
$$
\aligned
\ee\left(\bar{h}(c_i,c_j)|\FF_{j-1}\right)&=\ee\left(\bar{h}(c_i,c_j)|c_i\right)\\
&=\ee\left(\delta_{c_i,c_j}-p_{c_{i}}-p_{c_{j}}+p_{(2)}|c_i\right)\\
&=p_{c_{i}}-p_{c_{i}}-p_{(2)}+p_{(2)}=0,
\endaligned
$$
where
$$
\ee\left(\delta_{c_i,c_j}|c_i\right)=\pp(c_j=c_i)=p_{c_{i}}\ \
\text{and} \ \
\ee\left(p_{c_{j}}|c_i\right)=\ee p_{c_{j}}=\sum_{k=1}^Kp_k^2=p_{(2)}.
$$
Then we have
$$
\ee\left(z_{nj}|\FF_{j-1}\right)=0 \ \ \ \ \text{for}\ \ 2\leq j\leq n.
$$
Therefore, for each $n\geq2$, $\{z_{nj},2\le j\le n\}$ forms a martingale difference with respect to $\{\FF_j\}$. By (\ref{pp-1}), (\ref{pp-2}) and Lemma \ref{lem5}, we have
\beq\label{aa}
\aligned
&\ \ \ \ \sup_{-\infty<x<\infty}\left|\pp\left(\frac{Q_n-\mu_n}{\delta_n}\leq x\right)-\Phi(x)\right|\\
&\leq\sup_{-\infty<x<\infty}\big|\pp\left(T_n\leq x\right)-\Phi(x)\big|+\frac{a_n}{\sqrt{2\pi}}\\
&\ \ \ \ +\pp\left(\bigg|\frac{1}{2m\delta_n}\sum_{1\leq i<j\leq n}\frac{k_ik_j}{m}\overline{h}(c_i,c_j)
-\frac{1}{2m\delta_n}\sum_{i=1}^n\frac{k_i^2}{m}(p_{c_{i}}-p_{(2)})\bigg|>a_n\right),
\endaligned
\deq
where
\beq\label{c}
a_n=n^{-3/8}(\log n)^{-1/2}.
\deq
Now, we consider the first part, applying Lemma \ref{lem2-6} with the following martingale array
$$
T_n=\frac{1}{m\delta_n}\sum_{j=1}^n\sum_{i=1}^{j-1}A_{ij}\overline{h}(c_i,c_j)=\frac{1}{m\delta_n}\sum_{j=1}^nz_{nj}.
$$
Denote
$$
V_n^2=\frac{1}{m^2\delta_n^2}\sum_{j=1}^n\ee\left(z_{nj}^2|\FF_{j-1}\right).
$$
Firstly, we will check the condition (\ref{2-6-1}). By condition (\ref{3-1}) and the fact $\big|\overline{h}(c_i,c_j)\big|\le2$, we can derive
\begin{align*}
\frac{\max_{1\le j\le n}z_{nj}}{m\delta_n}&=\max_{1\le j\le n}\Bigg|\sum_{i=1}^{j-1}\frac{A_{ij}\overline{h}(c_i,c_j)}{m\delta_n}\Bigg|
\le 2\max_{1\le j\le n}\sum_{i=1}^{j-1}\frac{A_{ij}}{m\delta_n}\\
&\le 2\frac{\max_{1\le j\le n}k_j}{m\delta_n}\le Hn^{-1/2}.
\end{align*}
Next, we will check the condition (\ref{2-8}). It follows that
\begin{align}
V_n^2
&= \nonumber \frac{1}{m^2\delta_n^2}\sum_{j=1}^n\ee\left(\left(\sum_{i=1}^{j-1}A_{ij}\overline{h}(c_i,c_j)\right)^2\bigg|\FF_{j-1}\right)\\
&= \nonumber \frac{1}{m^2\delta_n^2}\sum_{j=1}^n\sum_{i=1}^{j-1}
\ee\left(\left(A_{ij}\overline{h}(c_i,c_j)\right)^2\big|\FF_{j-1}\right)\\
&\ \ \  \nonumber + \frac{2}{m^2\delta_n^2}\sum_{j=1}^n\sum_{1\le i<l\le j-1}\ee\left(\left(A_{ij}\overline{h}(c_i,c_j)\right)\left(A_{lj}\overline{h}(c_{l},c_j)\right)|\FF_{j-1}\right)\\
&=\nonumber \frac{1}{m^2\delta_n^2}\sum_{i=1}^{n-1}\sum_{j=i+1}^{n}
A_{ij}^2\ee\left(\left(\overline{h}(c_i,c_j)\right)^2\big|\FF_{j-1}\right)\\
&\ \ \  + \frac{2}{m^2\delta_n^2}\sum_{1\le i<l\le n-1}\sum_{j=l+1}^nA_{ij}A_{lj}\ee\left(\overline{h}(c_i,c_j)\overline{h}(c_{l},c_j)|\FF_{j-1}\right). ~\label{S1-0}
\end{align}
Since $A_{ij}\in\{0,1\}$, we have $A_{ij}^2=A_{ij}$, which together with (\ref{e-e}) and (\ref{b-1}) yield that
\begin{align}
&\ \nonumber\ee\left[\frac{1}{m^2\delta_n^2}\sum_{i=1}^{n-1}\sum_{j=i+1}^{n}A_{ij}^2
\ee\left(\left(\overline{h}(c_i,c_j)\right)^2\big|\FF_{j-1}\right)\right]\\
 =& \nonumber \frac{1}{m^2\delta_n^2} \sum_{i=1}^{n-1}\sum_{j=i+1}^{n}A_{ij}\ee\left[\ee\left(\left(\overline{h}(c_i,c_j)\right)^2\big|\FF_{j-1}\right)\right]\\
 =& \nonumber \frac{1}{m^2\delta_n^2}\sum_{i=1}^{n-1}\sum_{j=i+1}^{n}A_{ij}
 \ee\left[\ee\left(\left(\delta_{c_i,c_j}-p_{c_{i}}-p_{c_{j}}+p_{(2)}\right)^2\big|\FF_{j-1}\right)\right]\\
 =& \nonumber \frac{1}{m^2\delta_n^2}\sum_{i=1}^{n-1}\sum_{j=i+1}^{n}A_{ij}
 \ee\left[\ee\left(\left(\delta_{c_i,c_j}-p_{c_{i}}\right)^2+\left(p_{(2)}-p_{c_{j}}\right)^2 \right.\right.\\
 &\ \ \ \nonumber \left.\left.
 +2\left(\delta_{c_i,c_j}-p_{c_{i}}\right)\left(p_{(2)}-p_{c_{j}}\right)|\FF_{j-1}\right)\right]\\
 =& \nonumber\frac{1}{m^2\delta_n^2}\sum_{i=1}^{n-1}\sum_{j=i+1}^{n}A_{ij}
 \ee\left[\ee\left(\delta_{c_i,c_j}^2+p_{c_{i}}^2-2\delta_{c_i,c_j}p_{c_{i}}+p_{(2)}^2+p_{c_{j}}^2-2p_{(2)}p_{c_{j}} \right.\right.\\
 &\ \ \ \nonumber\left.\left.
 +2\delta_{c_i,c_j}p_{(2)}-2\delta_{c_i,c_j}p_{c_{j}}-2p_{(2)}p_{c_{i}}+2p_{c_{i}}p_{c_{j}}|\FF_{j-1}\right)\right]\\
=&\nonumber\frac{1}{m^2\delta_n^2}\sum_{i=1}^{n-1}\sum_{j=i+1}^{n}A_{ij}\ee\left[ p_{c_{i}}-3p_{c_{i}}^2+2p_{(2)}p_{c_i}-p_{(2)}^2+p_{(3)}\right]\\
=&\nonumber\frac{1}{2m^2\delta_n^2}\left(p_{(2)}+p_{(2)}^2-2p_{(3)}\right)\sum_{1\le i, j\le n}A_{ij}\\
 =&\frac{2m}{2m^2\delta_n^2}\left(p_{(2)}+p_{(2)}^2-2p_{(3)}\right)=1. ~\label{s1-1}
\end{align}
Moreover, for $1\le i<l<j\le n$, by (\ref{e-e}) and (\ref{b-1}), we deduce
\begin{align}
&\ \ \ \ \nonumber\ee\left[\overline{h}(c_i,c_j)\overline{h}(c_{l},c_j)\right]\\
&=\nonumber\ee\left[\ee\left[\left(\delta_{c_i,c_j}-p_{c_i}-p_{c_{j}}+p_{(2)}\right)
\left(\delta_{c_{l},c_j}-p_{c_{l}}-p_{c_{j}}+p_{(2)}\right)|\FF_{j-1}\right]\right]\\
&=\nonumber\ee\left[\ee\Big[\delta_{c_i,c_j}\delta_{c_{l},c_j}-\delta_{c_i,c_j}p_{c_{l}}
-\delta_{c_i,c_j}p_{c_{j}}+\delta_{c_i,c_j}p_{(2)}-\delta_{c_{l},c_j}p_{c_i}+p_{c_i}p_{c_{l}}+p_{c_i}p_{c_{j}} -p_{c_i}p_{(2)} \right.\\
 &\ \ \ \ \ \ \ \ \nonumber\left.  -\delta_{c_{l},c_j}p_{c_{j}}+p_{c_{j}}p_{c_{l}}+p_{c_{j}}^2-p_{c_{j}}p_{(2)}
 +\delta_{c_{l},c_j}p_{(2)}-p_{(2)}p_{c_{l}}-p_{(2)}p_{c_j}+p_{(2)}^2
 |\FF_{j-1}\Big]\right]\\
&=\nonumber\ee\Big[\ee\Big(\delta_{c_i,c_{l}}\frac{p_{c_i}+p_{c_{l}}}{2} -p_{c_i}p_{c_{l}}-p_{c_i}^2+p_{c_i}p_{(2)}-p_{c_i}p_{c_{l}}+p_{c_i}p_{c_{l}}+p_{c_i}p_{(2)}-p_{c_i}p_{(2)} \\
 & \ \ \ \ \ \ \ \ \nonumber \ -p_{c_{l}}^2+p_{c_{l}}p_{(2)}+p_{(3)}-p_{(2)}^2+p_{c_{l}}p_{(2)}
 -p_{c_{l}}p_{(2)}-p_{(2)}^2+p_{(2)}^2|c_i\Big)\Big]\\
&=\nonumber\ee\left[\ee\left(
\delta_{c_i,c_{l}}\frac{p_{c_i}+p_{c_{l}}}{2}-p_{c_i}p_{c_{l}}-p_{c_i}^2
+p_{c_i}p_{(2)}-p_{c_{l}}^2+p_{c_{l}}p_{(2)} +p_{(3)}-p_{(2)}^2|c_i\right)\right]=0,
\end{align}
which implies that
\beq\label{s1-2}
\aligned
&\ \ee\left[\frac{2}{m^2\delta_n^2}\sum_{1\le i<l\le n-1}\sum_{j=l+1}^nA_{ij}A_{lj}\ee\left(\overline{h}(c_i,c_j)\overline{h}(c_{l},c_j)|\FF_{j-1}\right)\right]\\
=&\ \frac{2}{m^2\delta_n^2}\sum_{1\le i<l\le n-1}
\sum_{j=l+1}^nA_{ij}A_{lj}\ee\left[\ee\left(\overline{h}(c_i,c_j)\overline{h}(c_{l},c_j)|\FF_{j-1}\right)\right]\\
=&\ \frac{2}{m^2\delta_n^2}\sum_{1\le i<l\le n-1}\sum_{j=l+1}^nA_{ij}A_{lj}\ee\left(\overline{h}(c_i,c_j)\overline{h}(c_{l},c_j)\right)\\
=&\ 0.
\endaligned
\deq
Substituting (\ref{s1-1}) and (\ref{s1-2}) into (\ref{S1-0}), we have $\ee V_n^2=1$.
It follows, from Markov's inequality, that
\beq\label{s1-3}
\pp\left(|V_{n}^2-1|>Hn^{-1/2}(\log n)^2\right)\le H\frac{n}{(\log n)^4}\ee\left[\left(V_{n}^2-1\right)^2\right].
\deq
Note that
\begin{align}
\ee\left[\left(V_{n}^2-1\right)^2\right]
&=\nonumber\ee\left[\left(\frac{1}{m^2\delta_n^2}\sum_{j=1}^n\ee\left(z_{nj}^2|\FF_{j-1}\right)-1\right)^2\right]\\
&=\nonumber\frac{1}{m^4\delta_n^4}Var\left(\sum_{j=1}^n\ee\left(z_{nj}^2|\FF_{j-1}\right)-m^2\delta_n^2\right)\\
&=\nonumber\frac{1}{m^4\delta_n^4}Var\left(\sum_{j=1}^n\ee\left(\left(\sum_{i=1}^{j-1}
A_{ij}\overline{h}(c_i,c_j)\right)^2\bigg|\FF_{j-1}\right)\right)\\
&=\nonumber\frac{1}{m^4\delta_n^4}Var\left(\sum_{j=1}^n\ee\left(\sum_{i_1=1}^{j-1}\sum_{i_2=1}^{j-1}A_{i_{1}j}A_{i_{2}j}
\overline{h}(c_{i_{1}},c_j)\overline{h}(c_{i_{2}},c_j)\bigg|\FF_{j-1}\right)\right)\\
&=\nonumber\frac{1}{m^4\delta_n^4}Var\left(\sum_{j=1}^n\sum_{i_1=1}^{j-1}\sum_{i_2=1}^{j-1}A_{i_{1}j}A_{i_{2}j}\ee\left(
\overline{h}(c_{i_{1}},c_j)\overline{h}(c_{i_{2}},c_j)|\FF_{j-1}\right)\right)\\
&=\nonumber\frac{1}{m^4\delta_n^4}Var\left(2\sum_{j=1}^n\sum_{1\le i_1<i_2<j-1}A_{i_{1}j}A_{i_{2}j}
\ee\left(\overline{h}(c_{i_{1}},c_j)\overline{h}(c_{i_{2}},c_j)|\FF_{j-1}\right)\right.\\
&\ \ \ \ \nonumber + \left.
\sum_{j=1}^n\sum_{i=1}^{j-1}A_{ij}\ee\left(\left(\overline{h}(c_i,c_j)\right)^2\big|\FF_{j-1}\right)\right)\\
&=\nonumber\frac{1}{m^4\delta_n^4}Var\left(2\sum_{1\le i_1<i_2<n-1}\left(\sum_{j=i_2+1}^nA_{i_{1}j}A_{i_{2}j}\right)
\ee\left(\overline{h}(c_{i_{1}},c_j)\overline{h}(c_{i_{2}},c_j)|\FF_{j-1}\right)\right.\\
&\ \ \ \ \nonumber + \left.\sum_{i=1}^{n-1}\left(\sum_{j=i+1}^{n}A_{ij}\right)
\ee\left(\left(\overline{h}(c_i,c_j)\right)^2\big|\FF_{j-1}\right)\right)\\
&=\nonumber\frac{4}{m^4\delta_n^4}\sum_{1\le i_1<i_2<n-1}\left(\sum_{j=i_2+1}^nA_{i_{1}j}A_{i_{2}j}\right)^2
Var\left(\ee\left(\overline{h}(c_{i_{1}},c_j)\overline{h}(c_{i_{2}},c_j)|\FF_{j-1}\right)\right)\\
&\ \ \ \ + \frac{1}{m^4\delta_n^4}\sum_{i=1}^{n-1}\left(\sum_{j=i+1}^{n}A_{ij}\right)^2
Var\left(\ee\left(\left(\overline{h}(c_i,c_j)\right)^2\big|\FF_{j-1}\right)\right), ~\label{s1-4}
\end{align}
where the last line follows by $\ee\left(\overline{h}(c_{i_{1}},c_j)\overline{h}(c_{i_{2}},c_j)|\FF_{j-1}\right)$ and $\ee(\left(\overline{h}(c_i,c_j)\right)^2|\FF_{j-1})$ are orthogonal.
Putting (\ref{3-1-1}), (\ref{s1-3}), (\ref{s1-4}) and the fact $|\overline{h}(c_i,c_j)|\le 2$ together, we obtain
\beq\label{i-9}
\aligned
&\ \ \ \ \pp\left(|V_{n}^2-1|>Hn^{-1/2}(\log n)^2\right)\\
&\le H\frac{n}{(\log n)^4m^4\delta_n^4}\left(\sum_{1\le i_1<i_2<n-1}\left(\sum_{j=i_2+1}^nA_{i_{1}j}A_{i_{2}j}\right)^2
+\sum_{i=1}^{n-1}\left(\sum_{j=i+1}^{n}A_{ij}\right)^2\right)\\
&\le H\frac{n}{(\log n)^4m^4\delta_n^4}\sum_{1\le i_1,i_2\le n}\left(\sum_{j=1}^nA_{i_{1}j}A_{i_{2}j}\right)^2\\
&\le Hn^{-1/4}\log n.
\endaligned
\deq
Hence, it is straightforward that
\beq\label{h-1}
\sup_{-\infty<x<\infty}\big|\pp(T_{n}\le x)-\Phi(x)\big|\le Hn^{-1/4}\log n.
\deq
Now, we consider the last past,
\begin{align*}
&\ \ \ \ \pp\left(\bigg|\frac{1}{2m\delta_n}\sum_{1\leq i<j\leq n}\frac{k_ik_j}{m}\overline{h}(c_i,c_j)
-\frac{1}{2m\delta_n}\sum_{i=1}^n\frac{k_i^2}{m}(p_{c_{i}}-p_{(2)})\bigg|>a_n\right)\\
&\leq\pp\left(\bigg|\frac{1}{2m\delta_n}\sum_{1\leq i<j\leq n}\frac{k_ik_j}{m}\overline{h}(c_i,c_j)\bigg|>\frac{a_n}{2}\right)\\
&\ \ \ +\pp\left(\bigg|\frac{1}{2m\delta_n}\sum_{i=1}^n\frac{k_i^2}{m}(p_{c_{i}}-p_{(2)})\bigg|>\frac{a_n}{2}\right).
\end{align*}
Together with (\ref{e-e}) and (\ref{b-1}), we can verify that
\begin{align}
Var(p_{c_{i}}-p_{(2)})
&=\nonumber\ee\left[\left(p_{c_{i}}-p_{(2)}\right)^2\right]
-\left[\ee\left(p_{c_{i}}-p_{(2)}\right)\right]^2\\
&=\nonumber\ee\left(p_{c_{i}}^2+p_{(2)}^2-2p_{c_{i}}p_{(2)}\right)-\left(p_{(2)}-p_{(2)}\right)^2\\
&=\nonumber p_{(3)}+p_{(2)}^2-2p_{(2)}^2\\
&= p_{(3)}-p_{(2)}^2 ~\label{b1-0}
\end{align}
and
\begin{align}
Var(\overline{h}(c_i,c_j))
&=\nonumber\ee\left[\left(\overline{h}(c_i,c_j)\right)^2\right]
-\left[\ee\left(\overline{h}(c_i,c_j)\right)\right]^2\\
&=\nonumber\ee\left[\ee\left(\left(\overline{h}(c_i,c_j)\right)^2\big|\FF_{j-1}\right)\right]\\
&=\nonumber\ee\left[\ee\left(\left(\delta_{c_i,c_j}-p_{c_{i}}-p_{c_{j}}+p_{(2)}\right)^2\big|\FF_{j-1}\right)\right]\\
&= p_{(2)}+p_{(2)}^2-2p_{(3)}. \label{b1-1}
\end{align}
From (\ref{3-1}) and (\ref{c}), we get
\beq\label{b1-2}
\frac{1}{a_n^2m^3}\left(\sum_{i=1}^nk_i^2\right)^2\le \frac{\max_{1\le i\le n}k_i^2}{a_n^2m}
=\frac{1}{a_n^2}\left(\frac{\max_{1\le i\le n}k_i}{\sqrt{m}}\right)^2
\le Hn^{-1/4}\log n.
\deq
Combining (\ref{b1-0}), (\ref{b1-1}), (\ref{b1-2}) and Markov's inequality together, it holds that
\begin{align}
\pp\left(\bigg|\frac{1}{2m\delta_n}\sum_{1\leq i<j\leq n}\frac{k_ik_j}{m}\overline{h}(c_i,c_j)\bigg|>\frac{a_n}{2}\right)
\nonumber&\leq \frac{4}{a_n^2}\ee\left(\frac{1}{2m\delta_n}\sum_{1\leq i<j\leq n}\frac{k_ik_j}{m}\overline{h}(c_i,c_j)\right)^2\\
\nonumber&=\frac{1}{a_n^2m^2\delta_n^2}\sum_{1\leq i<j\leq n}\frac{k_i^2k_j^2}{m^2}Var(\overline{h}(c_i,c_j))\\
\nonumber&\le H\frac{\left(\sum_{i=1}^nk_i^2\right)^2}{a_n^2m^3}\\
&\le Hn^{-1/4}\log n ~\label{i-10}
\end{align}
and
\begin{align}
\pp\left(\bigg|\frac{1}{2m\delta_n}\sum_{i=1}^n\frac{k_i^2}{m}(p_{c_{i}}-p_{(2)})\bigg|>\frac{a_n}{2}\right)
\nonumber&\leq \frac{4}{a_n^2}\ee\left(\frac{1}{2m\delta_n}\sum_{i=1}^n\frac{k_i^2}{m}(p_{c_{i}}-p_{(2)})\right)^2\\
\nonumber&=\frac{1}{a_n^2m^2\delta_n^2}\sum_{i=1}^n\frac{k_i^4}{m^2}Var(p_{c_{i}}-p_{(2)})\\
\nonumber&\le H\frac{\left(\sum_{i=1}^nk_i^2\right)^2}{a_n^2m^3}\\
&\le Hn^{-1/4}\log n.~\label{i-11}
\end{align}
which, together with (\ref{aa}), (\ref{c}) and (\ref{h-1}), implies that
\begin{align*}
\sup_{-\infty<x<\infty}\bigg|\pp\left(\frac{Q_n-\mu_n}{\delta_n}\leq x\right)-\Phi(x)\bigg|
\le Hn^{-1/4}\log n.
\end{align*}
\end{proof}

\begin{proof}[Proof of Corollary \ref{cor2-1}]
Applying (\ref{2-3}), (\ref{2-3-1}) and (\ref{cc}) together, we have
\begin{align*}
\frac{Q_n-\mu_n}{\sigma_n}
&=\frac{1}{m\sigma_n}\sum_{1\leq i<j\leq n}B_{ij}\bar{h}(c_i,c_j)-\frac{1}{m\sigma_n}\sum_{i=1}^nB_{ii}(p_{c_{i}}-p_{(2)})\\
&=\frac{1}{m\sigma_n}\sum_{1\leq i<j\leq n}A_{ij}\overline{h}(c_i,c_j)-\frac{1}{2m\sigma_n}\sum_{1\leq i<j\leq n}\frac{k_ik_j}{m}\overline{h}(c_i,c_j)\\
&\ \ \ +\frac{1}{2m\sigma_n}\sum_{i=1}^n\frac{k_i^2}{m}(p_{c_{i}}-p_{(2)}).
\end{align*}
Let
$$
E_n=\frac{1}{m\sigma_n}\sum_{j=1}^n\sum_{i=1}^{j-1}A_{ij}\overline{h}(c_i,c_j)=\frac{1}{m\sigma_n}\sum_{j=1}^nz_{nj},
$$
where $z_{n1}=0$ and $z_{nj}=\sum_{i=1}^{j-1}A_{ij}\overline{h}(c_i,c_j)$ for $2\le j\le n$. Since $\{z_{nj}\}$ forms a martingale difference with respect to $\{\FF_i\}$. Then by Lemma \ref{lem5}, we have
\beq\label{i-a}
\aligned
&\ \ \ \ \sup_{-\infty<x<\infty}\left|\pp\left(\frac{Q_n-\mu_n}{\sigma_n}\leq x\right)-\Phi(x)\right|\\
&\leq\sup_{-\infty<x<\infty}\big|\pp\left(E_n\leq x\right)-\Phi(x)\big|+\frac{a_n}{\sqrt{2\pi}}\\
&\ \ \ \ +\pp\left(\bigg|\frac{1}{2m\sigma_n}\sum_{1\leq i<j\leq n}\frac{k_ik_j}{m}\overline{h}(c_i,c_j)
-\frac{1}{2m\sigma_n}\sum_{i=1}^n\frac{k_i^2}{m}(p_{c_{i}}-p_{(2)})\bigg|>a_n\right),
\endaligned
\deq
where
\beq\label{i-a1}
a_n=n^{-3/8}(\log n)^{-1/2}.
\deq
We consider the first part, applying Lemma \ref{lem2-6} with the following martingale array
$$
E_n=\frac{1}{m\sigma_n}\sum_{j=1}^n\sum_{i=1}^{j-1}A_{ij}\overline{h}(c_i,c_j)=\frac{1}{m\sigma_n}\sum_{j=1}^nz_{nj}.
$$
Denote
$$
L_n^2=\frac{1}{m^2\sigma_n^2}\sum_{j=1}^n\ee\left(z_{nj}^2|\FF_{j-1}\right)
$$
and
$$
r_1=p_{(2)}+p_{(2)}^2-2p_{(3)}, \ \ \ r_2=p_{(3)}-p_{(2)}^2.
$$
Note that by (\ref{2-3-1}) and (\ref{b1}), we deduce
\beq\label{i-1}
\frac{\delta_n^2}{\sigma_n^2}=\frac{r_1}{m\left(\frac{r_1}{2m^2}\sum_{1\leq i\neq j\leq n}B_{ij}^2+\frac{r_2}{m^2}\sum_{i=1}^nB_{ii}^2\right)}
=\frac{r_1}{\frac{r_1}{2m}\sum_{1\leq i\neq j\leq n}B_{ij}^2+\frac{r_2}{m}\sum_{i=1}^nB_{ii}^2},
\deq
then
\beq\label{i-6}
\aligned
\left|\frac{\delta_n^2}{\sigma_n^2}-1\right|
&=\left|\frac{r_1-\frac{r_1}{2m}\sum_{1\leq i\neq j\leq n}B_{ij}^2-\frac{r_2}{m}\sum_{i=1}^nB_{ii}^2}
{\frac{r_1}{2m}\sum_{1\leq i\neq j\leq n}B_{ij}^2+\frac{r_2}{m}\sum_{i=1}^nB_{ii}^2}\right|\\
&=\left|\frac{r_1\left(1-\frac{1}{2m}\sum_{1\leq i\neq j\leq n}B_{ij}^2\right)
-\frac{r_2}{m}\sum_{i=1}^nB_{ii}^2}{\frac{r_1}{2m}\sum_{1\leq i\neq j\leq n}B_{ij}^2
+\frac{r_2}{m}\sum_{i=1}^nB_{ii}^2}\right|\\
&\le \frac{\frac{r_1}{2m}\left|2m-\sum_{1\leq i\neq j\leq n}B_{ij}^2\right|
+\frac{r_2}{m}\sum_{i=1}^nB_{ii}^2}{\frac{r_1}{2m}\sum_{1\leq i\neq j\leq n}B_{ij}^2}.
\endaligned
\deq
It follows, from (\ref{h1}), that
\begin{align*}
\sum_{1\leq i\neq j\leq n}B_{ij}^2
&=\sum_{1\leq i\neq j\leq n}\left(A_{ij}^2+\frac{k_i^2k_j^2}{4m^2}-2A_{ij}\frac{k_ik_j}{2m}\right)\\
&=\sum_{1\leq i\neq j\leq n}A_{ij}+\sum_{1\leq i\neq j\leq n}\frac{k_i^2k_j^2}{4m^2}
-\sum_{1\leq i\neq j\leq n}2A_{ij}\frac{k_ik_j}{2m}\\
&=2m+\sum_{1\leq i\neq j\leq n}\frac{k_i^2k_j^2}{4m^2}
-\sum_{1\leq i\neq j\leq n}A_{ij}\frac{k_ik_j}{m},
\end{align*}
which together with (\ref{h1}), (\ref{3-1}) and the Cauchy-Schwarz inequality yield that
\beq\label{i-7}
\aligned
& \ \ \ \ \frac{r_1}{2m}\left|2m-\sum_{1\leq i\neq j\leq n}B_{ij}^2\right|+\frac{r_2\sum_{i=1}^nB_{ii}^2}{m}\\
&\le \frac{r_1}{2m}\sum_{1\leq i\neq j\leq n}\frac{k_i^2k_j^2}{4m^2}
+\frac{r_1}{2m}\sum_{1\leq i\neq j\leq n}A_{ij}\frac{k_ik_j}{m}+\frac{r_2\sum_{i=1}^nB_{ii}^2}{m}\\
&\le \frac{r_1\left(\sum_{i=1}^nk_i^2\right)^2}{8m^3}+\frac{r_1}{2m^2}\sqrt{\sum_{1\leq i\neq j\leq n}A_{ij}}\sqrt{\sum_{1\leq i\neq j\leq n}k_i^2k_j^2}+\frac{r_2}{m}\sum_{i=1}^n\frac{k_i^4}{4m^2}\\
&\le \frac{r_1\left(\sum_{i=1}^nk_i^2\right)^2}{8m^3}
+\sqrt{\frac{2r_1^2\left(\sum_{i=1}^nk_i^2\right)^2}{m^3}}
+\frac{r_2\left(\sum_{i=1}^nk_i^2\right)^2}{4m^3}\\
&\le \frac{r_1}{2}\frac{\max_{1\le j\le n}k_j^2}{m}+\sqrt{2}r_1\frac{\max_{1\le j\le n}k_j}{\sqrt{m}}+r_2\frac{\max_{1\le j\le n}k_j^2}{m}\\
&\le H\frac{\max_{1\le j\le n}k_j}{\sqrt{m}}\le Hn^{-1/2}.
\endaligned
\deq
Since $Var(\bar{h}(c_1,c_2))=p_{(2)}+p_{(2)}^2-2p_{(3)}=r_1$ and $Var(p_{c_{1}}-p_{(2)})=p_{(3)}-p_{(2)}^2=r_2$, then from (\ref{i-7}), we obtain
\beq\label{i-31}
\lim_{n\to\infty}\frac{\sum_{1\leq i\neq j\leq n}B_{ij}^2}{2m}=1 \ \ \ \text{and} \ \ \ \lim_{n\to\infty}\frac{\sum_{i=1}^nB_{ii}^2}{m}=0.
\deq
Substituting (\ref{i-31}) into (\ref{i-1}) yields
$$
\lim_{n\to\infty}\frac{\delta_n^2}{\sigma_n^2}=1,
$$
then for any $\varepsilon>0$, there exists a positive integer $N$, such that for every $n>N$, we deduced that
$$
\left|\frac{\delta_n^2}{\sigma_n^2}-1\right|<\varepsilon.
$$
Therefore, it is straightforward that
\beq\label{i-22}
\frac{\delta_n^2}{\sigma_n^2}\le \max\left\{1+\varepsilon,\frac{\delta_i^2}{\sigma_i^2},1\le i\le N\right\}.
\deq
Firstly, we will check the condition (\ref{2-6-1}) of Lemma \ref{lem2-6}. By applying (\ref{3-1}), (\ref{i-22}) and the fact $\big|\overline{h}(c_i,c_j)\big|\le2$, we can derive
\begin{align*}
\frac{\max_{1\le j\le n}z_{nj}}{m\sigma_n}
&=\frac{\delta_n}{\sigma_n}\frac{\max_{1\le j\le n}z_{nj}}{m\delta_n}
=\frac{\delta_n}{\sigma_n}\max_{1\le j\le n}\Bigg|\sum_{i=1}^{j-1}\frac{A_{ij}\overline{h}(c_i,c_j)}{m\delta_n}\Bigg|\\
&\le \frac{2\delta_n}{\sigma_n}\max_{1\le j\le n}\sum_{i=1}^{j-1}\frac{A_{ij}}{m\delta_n}
\le H\frac{\max_{1\le j\le n}k_j}{m\delta_n}\le Hn^{-1/2}.
\end{align*}
From (\ref{i-31}), then for any $\varepsilon>0$, there exists a positive integer $N$, such that for every $n>N$, we have
$$
\left|\frac{\sum_{1\leq i\neq j\leq n}B_{ij}^2}{2m}-1\right|<\varepsilon,
$$
which implies that
\beq\label{i-4}
\frac{\sum_{1\leq i\neq j\leq n}B_{ij}^2}{2m}\ge \min\left\{1-\varepsilon,\frac{\sum_{1\leq i\neq j\leq k}B_{ij}^2}{2m},1\le k\le N\right\}.
\deq
Combining (\ref{i-6}), (\ref{i-7}) and (\ref{i-4}), it holds that
\beq\label{i-8}
\left|\frac{\delta_n^2}{\sigma_n^2}-1\right|\le Hn^{-1/2}.
\deq
Next, we will check the condition (\ref{2-8}) of Lemma \ref{lem2-6}. It follows, from (\ref{i-9}), (\ref{i-22}) and (\ref{i-8}), that
\begin{align*}
&\ \ \ \ \pp\left(|L_{n}^2-1|>Hn^{-1/2}(\log n)^2\right)\\
&=\pp\left(\left|\frac{1}{m^2\sigma_n^2}\sum_{j=1}^n\ee\left(z_{nj}^2|\FF_{j-1}\right)-1\right|
>Hn^{-1/2}(\log n)^2\right)\\
&=\pp\left(\left|\frac{\delta_n^2}{\sigma_n^2}\frac{1}{m^2\delta_n^2}
\sum_{j=1}^n\ee\left(z_{nj}^2|\FF_{j-1}\right)-\frac{\delta_n^2}{\sigma_n^2}
+\frac{\delta_n^2}{\sigma_n^2}-1\right|>Hn^{-1/2}(\log n)^2\right)\\
&\le \pp\left(\frac{\delta_n^2}{\sigma_n^2}\left|\frac{1}{m^2\delta_n^2}
\sum_{j=1}^n\ee\left(z_{nj}^2|\FF_{j-1}\right)-1\right|
+\left|\frac{\delta_n^2}{\sigma_n^2}-1\right|>Hn^{-1/2}(\log n)^2\right)\\
&=\pp\left(\frac{\delta_n^2}{\sigma_n^2}\left|V_n^2-1\right|
+\left|\frac{\delta_n^2}{\sigma_n^2}-1\right|>Hn^{-1/2}(\log n)^2\right)\\
&\le \pp\left(\left|V_n^2-1\right|>Hn^{-1/2}(\log n)^2\right)\\
&\le Hn^{-1/4}\log n.
\end{align*}
Hence, we deduce
\beq\label{i-a2}
\sup_{-\infty<x<\infty}\big|\pp(E_{n}\le x)-\Phi(x)\big|\le Hn^{-1/4}\log n.
\deq
Now, we consider the last past,
\begin{align*}
&\ \ \ \ \pp\left(\bigg|\frac{1}{2m\sigma_n}\sum_{1\leq i<j\leq n}\frac{k_ik_j}{m}\overline{h}(c_i,c_j)
-\frac{1}{2m\sigma_n}\sum_{i=1}^n\frac{k_i^2}{m}(p_{c_{i}}-p_{(2)})\bigg|>a_n\right)\\
&\leq\pp\left(\bigg|\frac{1}{2m\sigma_n}\sum_{1\leq i<j\leq n}\frac{k_ik_j}{m}\overline{h}(c_i,c_j)\bigg|>\frac{a_n}{2}\right)\\
&\ \ \ +\pp\left(\bigg|\frac{1}{2m\sigma_n}\sum_{i=1}^n\frac{k_i^2}{m}
(p_{c_{i}}-p_{(2)})\bigg|>\frac{a_n}{2}\right).
\end{align*}
Applying (\ref{i-10}), (\ref{i-11}) and (\ref{i-22}), we have
\begin{align*}
\pp\left(\bigg|\frac{1}{2m\sigma_n}\sum_{1\leq i<j\leq n}\frac{k_ik_j}{m}\overline{h}(c_i,c_j)\bigg|>\frac{a_n}{2}\right)
&=\pp\left(\bigg|\frac{\delta_n}{\sigma_n}\frac{1}{2m\delta_n}\sum_{1\leq i<j\leq n}\frac{k_ik_j}{m}\overline{h}(c_i,c_j)\bigg|>\frac{a_n}{2}\right)\\
&\le \pp\left(\bigg|\frac{1}{2m\delta_n}\sum_{1\leq i<j\leq n}\frac{k_ik_j}{m}\overline{h}(c_i,c_j)\bigg|>H a_n \right)\\
&\le Hn^{-1/4}\log n
\end{align*}
and
\begin{align*}
\pp\left(\bigg|\frac{1}{2m\sigma_n}\sum_{i=1}^n\frac{k_i^2}{m}
(p_{c_{i}}-p_{(2)})\bigg|>\frac{a_n}{2}\right)
&=\pp\left(\bigg|\frac{\delta_n}{\sigma_n}\frac{1}{2m\delta_n}\sum_{i=1}^n\frac{k_i^2}{m}
(p_{c_{i}}-p_{(2)})\bigg|>\frac{a_n}{2}\right)\\
&\le \pp\left(\bigg|\frac{1}{2m\delta_n}\sum_{i=1}^n\frac{k_i^2}{m}
(p_{c_{i}}-p_{(2)})\bigg|>H a_n \right)\\
&\le Hn^{-1/4}\log n,
\end{align*}
which, together with (\ref{i-a}), (\ref{i-a1}) and (\ref{i-a2}), implies that
$$
\sup_{-\infty<x<\infty}\left|\pp\left(\frac{Q_n-\mu_n}{\sigma_n}\leq x\right)-\Phi(x)\right|
\le Hn^{-1/4}\log n.
$$
\end{proof}

\begin{proof}[Proof of Theorem \ref{thm2-2}]
Putting (\ref{cc}) and (\ref{2-3}) together, we deduce that
\begin{align}
\big|b_n\left(Q_n-\mu_n\right)\big|
&=\nonumber\Bigg|\frac{b_n}{m}\sum_{1\leq i<j\leq n}B_{ij}\bar{h}(c_i,c_j)-\frac{b_n}{m}\sum_{i=1}^nB_{ii}(p_{c_{i}}-p_{(2)})\Bigg|\\
&=\nonumber\Bigg|\frac{b_n}{m}\sum_{1\leq i<j\leq n}A_{ij}\overline{h}(c_i,c_j)-\frac{b_n}{m}\sum_{1\leq i<j\leq n}\frac{k_ik_j}{2m}\overline{h}(c_i,c_j)\\
&\ \ \ \nonumber +\frac{b_n}{m}\sum_{i=1}^n\frac{k_i^2}{2m}(p_{c_{i}}-p_{(2)})\Bigg|\\
&\le \nonumber\Bigg|\frac{b_n}{m}\sum_{1\leq i<j\leq n}A_{ij}\overline{h}(c_i,c_j)\bigg|+\bigg|\frac{b_n}{m}\sum_{1\leq i<j\leq
n}\frac{k_ik_j}{2m}\overline{h}(c_i,c_j)\Bigg|\\
&\ \ \  +\Bigg|\frac{b_n}{m}\sum_{i=1}^n\frac{k_i^2}{2m}(p_{c_{i}}-p_{(2)})\Bigg|. ~\label{c1-0}
\end{align}
From (\ref{3-1}), (\ref{p}) and the fact $|p_{c_{i}}-p_{(2)}|\le1$, for $n\to\infty$, we have
\beq\label{c1-1}
\Bigg|\frac{b_n}{m}\sum_{i=1}^n\frac{k_i^2}{2m}(p_{c_{i}}-p_{(2)})\Bigg|\le \frac{b_n}{m}\frac{2m\max_{1\le i\le n}k_i}{2m}\\
=\frac{b_n\max_{1\le i\le n}k_i}{m}\to0.
\deq
Hence, (\ref{e}) follows immediately from (\ref{c1-0}) and (\ref{c1-1}) if
$$
\frac{b_n}{m}\sum_{1\leq i<j\leq n}A_{ij}\overline{h}(c_i,c_j)\xrightarrow{a.s.}0
$$
and
$$
\frac{b_n}{m}\sum_{1\leq i<j\leq n}\frac{k_ik_j}{2m}\overline{h}(c_i,c_j)\xrightarrow{a.s.}0.
$$
Firstly, we need to verify the condition $x>8eD_n$ of Lemma \ref{lem2-3-1}. From the notation of Lemma \ref{lem2-3-1} and the fact $\big|\bar{h}(c_i,c_j)\big|\le 2$, we have
$$
\big|A_{ij}\bar{h}(c_i,c_j)\big|\le 2A_{ij}=:a_{ij},
$$
then
\beq\label{d1-0}
D_n=\left(\sum_{1\le i<j\le n}a_{ij}^2\right)^{1/2}=\left(\sum_{1\le i<j\le n}\left(2A_{ij}\right)^2\right)^{1/2}.
\deq
Putting (\ref{d1-0}) and (\ref{p}) together, for $n\to\infty$, we obtain
$$
\frac{mr}{8eb_nD_n}=\frac{mr}{8eb_n\left(\sum_{1\le i<j\le n}(2A_{ij})^2\right)^{1/2}}
=\frac{mr}{32eb_n\left(\sum_{1\le i<j\le n}A_{ij}^2\right)^{1/2}}\ge H\frac{\sqrt{m}}{b_n}
\to\infty.
$$
Then, the condition $mr/b_n>8eD_n$ holds. By using (\ref{p}), we deduce for any $r>0$,
\begin{align*}
\sum_{n=1}^{\infty}\pp\left(\frac{b_n}{m}\bigg|\sum_{1\leq i<j\leq n}A_{ij}\overline{h}(c_i,c_j)\bigg|>r\right)
&\le \sum_{n=1}^{\infty}\exp\left(-\frac{mr}{8eb_n\left(\sum_{1\le i<j\le n}A_{ij}^2\right)^{1/2}}\right)\\
&\le \sum_{n=1}^{\infty}\exp\left(-\frac{\sqrt{m}r}{8eb_n}\right)\\
&< \sum_{n=1}^{\infty}\exp(-t\log n)\\
&= \sum_{n=1}^{\infty}n^{-t}<\infty,
\end{align*}
where $t$ is a constant and $t>1$.
Using the Borel-Cantelli lemma, we have
$$
\frac{b_n}{m}\sum_{1\leq i<j\leq n}A_{ij}\overline{h}(c_i,c_j)\xrightarrow{a.s.}0.
$$
In addition, combining (\ref{3-1}), (\ref{p}) and Markov's inequality together, we obtain for any $r>0$,
\begin{align*}
\sum_{n=1}^{\infty}\pp\left(\frac{b_n}{m}\bigg|\sum_{1\le i<j\le n}\frac{k_ik_j}{2m}\overline{h}(c_i,c_j)\bigg|>r\right)
&\le \sum_{n=1}^{\infty}\frac{b_n^2}{m^2r^2}\ee\left(\sum_{1\leq i<j\leq n}\frac{k_ik_j}{2m}\overline{h}(c_i,c_j)\right)^2\\
&\le \sum_{n=1}^{\infty}\frac{b_n^2}{m^2r^2}\sum_{1\leq i<j\leq n}\frac{k_i^2k_j^2}{4m^2}Var\left(\overline{h}(c_i,c_j)\right)\\
&\le H\sum_{n=1}^{\infty}\frac{b_n^2}{m^2}\frac{\left(\sum_{i=1}^nk_i^2\right)^2}{m^2}\\
&\le H\sum_{n=1}^{\infty}\frac{b_n^2}{m}\frac{\max_{1\le i\le n}k_i^2}{m}\\
&< H\sum_{n=2}^{\infty}\frac{1}{n(\log n)^2}<\infty,
\end{align*}
which implies that
$$
\frac{b_n}{m}\sum_{1\leq i<j\leq n}\frac{k_ik_j}{2m}\overline{h}(c_i,c_j)\xrightarrow{a.s.}0.
$$
Hence, this completes the proof of Theorem \ref{thm2-2}.
\end{proof}

\end{document}